\documentclass[12pt, reqno]{amsart}
\usepackage{amsmath, amsthm, amscd, amsfonts, amssymb, graphicx, color}
\usepackage[bookmarksnumbered, colorlinks, plainpages]{hyperref}
\hypersetup{colorlinks=true,linkcolor=red, anchorcolor=green, citecolor=cyan, urlcolor=red, filecolor=magenta, pdftoolbar=true}

\newtheorem{theorem}{Theorem}[section]
\newtheorem{lemma}[theorem]{Lemma}
\newtheorem{proposition}[theorem]{Proposition}

\theoremstyle{definition}
\newtheorem{definition}[theorem]{Definition}

\theoremstyle{remark}

\numberwithin{equation}{section}

\begin{document}

\setcounter{page}{1}

\title[Disjointness]{Disjoint hypercyclicity of weighted translations}

\author[Y.X. Liang and Z.H. Zhou] {Yu-Xia Liang and Ze-Hua Zhou$^*$}

\address{\newline  Yu-Xia Liang\newline School of Mathematical Sciences,
Tianjin Normal University, Tianjin 300387, P.R. China.}
\email{liangyx1986@126.com}

\address{\newline Ze-Hua Zhou\newline School of Mathematics, Tianjin University, Tianjin 300354,
P.R. China.}
\email{zehuazhoumath@aliyun.com;zhzhou@tju.edu.cn}

\subjclass[2010]{Primary: 47A16; Secondary: 46E30, 43A15.}

\keywords{disjoint hypercyclic, weighted translation, $L^p$-space.}

\date{}
\thanks{\noindent $^*$Corresponding author.\\
The work was supported in part by the National Natural
Science Foundation of China (Grant Nos. 11771323, 11701422).}

\begin{abstract}
Given a locally compact group $G$ and $1\leq p<\infty$, a sufficient condition ensuring the \emph{disjoint hypercyclicity}  of finitely many weighted translations  on $L^p(G)$  was investigated in this paper.
\end{abstract} \maketitle

\section{Introduction and preliminaries}
Let $L(X)$ denote the space of all linear and continuous operators on a separable, infinite dimensional Banach space $X$.  A continuous linear operator $T \in L (X)$ is called \emph{hypercyclic} if there is an $f \in X$ such that the orbit $$Orb(T,f)=\{f,\;Tf,\;T^2f,\;\cdots,T^nf,\cdots\} $$ is dense in $X$, and in this case we refer to $f$ as a hypercyclic vector for $T$.  Moreover, for each pair $U, V$ of nonempty open subsets of $X$, if there exists $m\in \mathbb{N}$  such that $T^mU\cap V\neq \emptyset,$ then $T$ is called \emph{topologically transitive}. An operator $T \in L (X)$ is \emph{hypercyclic} if and only if $T$ is \emph{topologically transitive}. For motivation, examples and background about linear dynamics we refer the readers to the books \cite{BM0} by Bayart and Matheron, \cite{GM} by Grosse-Erdmann and Peris.

As in the papers \cite{C2,CC2}, let $G$ be a \emph{locally compact group} with identity $e$ and a right-invariant Haar measure $\lambda$. Here we assume that $G$ is \emph{second countable} and denote by $L^p(G)$ for $1\leq p<\infty$ the \emph{complex Lebesgue space} with respect to $\lambda$. A \emph{bounded continuous} function $w: G\rightarrow (0,\infty)$ is called a \emph{weight} on $G$. For $a\in G$, let $\delta_a$ be the unit point mass at $a$.   We define a \emph{weighted translation} on $G$ (that is, a weighted convolution operator) $T_{a,w}: L^p(G) \rightarrow L^p(G)$ by $$T_{a,w} f=wT_a(f),\;\;f\in L^p(G),$$ where $w$ is a \emph{weight} on $G$ and $T_a(f)=f*\delta_a\in L^p(G)$ is the convolution:
$$(f*\delta_a) (x):=\int_{G} f(xy^{-1}) d\delta_a(y)=f(xa^{-1}),\;x\in G,$$ which is the right translation of $f$ by $a^{-1}$. Furthermore, we denote $(f*\delta_a^s) (x)=f(xa^{-s}),\;x\in G.$
It follows $$T_{a,w}f(x)=w(x)f(xa^{-1}),\;\;f\in L^p(G),\;x\in G.$$
Given a compact subgroup $H$ of $G$, the right coset space $G/ H$ admits a $G$-invariant measure $\nu=\lambda\circ q^{-1},$ where $q: G\rightarrow G/ H$ is the quotient map. Convolution operators on the space $L^p(G/H) $ with respect to $\nu$ can be lifted naturally to $L^p(G)$ (see \cite[p 77]{Chu}). Hence all the following results in this paper for $L^p(G)$ can be extended  naturally to $L^p(G/H).$ Without loss of generality, here we will confine our discussion to $L^p(G)$.

It is clear that a translation operator $T_a$ is never \emph{hypercyclic} since $\|T_a\|\leq 1.$ However, a weighted translation operator $T_{a,w}$ can be \emph{hypercyclic}. An element $a$ in a group $G$ is said to be a \emph{torsion} element if it is of finite order, that is, there is an order $d\in \mathbb{N}$ such that $a^d=e$. In a locally compact group $G$, an element $a\in G$ is called \emph{periodic} \cite{GMo} if the closed subgroup $G(a)$ generated by $a$ is \emph{compact}. An element $a\in G$ is said to be \emph{aperiodic} if it is not \emph{periodic.} For discrete groups, \emph{periodic} and \emph{torsion} elements are identical, thus  \emph{aperiodic} elements are \emph{non-torsion}  elements, however, the converse is not true. For instance, the circle group $\mathbb{T}=\{z\in \mathbb{C}:\;|z|=1\}$ is compact and thus all elements in $\mathbb{T}$  are \emph{periodic}, but an element $e^{i\theta}\in \mathbb{T}$ is \emph{non-torsion} if $\theta/2\pi$ is irrational. As regards to the \emph{aperiodic}  element $a$ in a locally compact group $G$, we have the following lemma.
\begin{lemma}\label{lem aperiodic}\cite[Lemma 2.1]{CC2} An  element $a$ in a second countable group $G$ is \emph{aperiodic} if and only if  for each compact subset $K\subset G$, there exists $N\in \mathbb{N}$ such that $K\cap K a^n=\emptyset$ for $n>N.$\end{lemma}

 Recently,  Chen \cite{C2} gave a complete characterization for the \emph{hypercyclic weighted translations} generated by \emph{non-torsion} elements; we include the theorem below for the convenience of the readers.\\
\textbf{Theorem A}\;\cite[Theorem 2.1]{C2} Let $G$ be a locally compact group, and let $a$ be a \emph{non-torsion} element in $G.$ Let $1\leq p< \infty$ and $T_{a,w}$ be a weighted translation on $L^p(G)$. The following conditions are equivalent:\\
$(i)$ \;$T_{a,w}$ is hypercyclic.\\
$(ii)$\; For each compact subset $K\subset G$ with $\lambda(K)>0$, there is a sequence of Borel sets $(E_k)$ in $K$ such that $\lambda(K)=
\lim\limits_{k\rightarrow \infty}\lambda(E_k)$ and both sequences
$$\varphi_n:=\prod_{s=1}^n w* \delta_{a^{-1}}^s\;\;\;\;\mbox{and}\;\;\;\; \tilde{\varphi}_n:=\left(\prod_{s=0}^{n-1} w*\delta_a^s\right)^{-1}$$
admit respectively subsequences $(\varphi_{n_k})$ and $(\tilde{\varphi}_{n_k})$ satisfying
$$\lim\limits_{k\rightarrow \infty} \|\varphi_{n_k} \Big|_{E_k}\|_\infty=\lim\limits_{k\rightarrow \infty}\|\tilde{\varphi}_{n_k} \Big|_{E_k}\|_\infty =0.$$
Here we note $\|f\|_{\infty }\equiv \inf\{C\geq 0:|f(x)|\leq C{\text{ for almost every }}x\}.$ In the sequel, we define a self-map $S_{a,w}$ on the subspace $L_c^p(G)$ of functions in $L^p(G)$ with \emph{compact support} by $$ S_{a,w}\;(h)=\frac{h}{w}* \delta_{a^{-1}},\;\;h\in L_c^p(G).$$ It's easy to show that $ T_{a,w} S_{a,w} h=h \;\;\mbox{for}\;\;h\in L_c^p(G).$

In 2009, Chen and Chu characterized the hypercyclic weighted translations by \emph{aperiodic} elements in coset spaces of discrete groups in \cite{CC1}, which was extended to all \emph{non-discrete} coset spaces in \cite{CC2}. At the same time, the conditions for a weighted translation on $L^p(G)$ to be chaotic were obtained too. After that, in 2013, Chen generalized the results in  \cite{CC2} to weighted translations, generated by \emph{non-torsion} elements in \cite{C2} (see Theorem A).  Inspired by the above work, we are interested in this paper in characterizing d-hypercyclicity (see Section 2) of $N\geq 2$ weighted translations acting on $L^p(G)$.   To be precise, given $N\geq 2$ \emph{bounded continuous} functions $w_i: G\rightarrow (0,\infty)$ and $N$ weighted translations $T_{a,w_1},\cdots, T_{a,w_N}$ on $L^p(G)$ were defined as $$T_{a,w_i} (f)(x)=w_i(x)f(xa^{-1}),\;\;f\in L^p(G)\;\;\mbox{and}\;\;x\in G,\;\;\mbox{for}\;\;i=1,\cdots,N.$$

Based on the above foundations, we will extend the previous result from a single  weighted translation  to  $N\geq 2$ weighted translations acting on $L^p(G)$. This paper is organized as follows:   some basic definitions and propositions  are given in section 2; While section 3 is devoted to the main statements.

\section{Auxiliary results}
In this section, we list some auxiliary results which will be used later.
\begin{definition}\label{defn d-hypercyclic}\cite[Definition 1.3.1]{M1} We say that $N\geq 2$  \emph{hypercyclic} operators $T_1,...,T_N$ in $L(X)$ are  \emph{disjoint hypercyclic} or \emph{d-hypercyclic} provided that there is some $f\in X$ such that the vector $(f,...,f)\in X^N$ is hypercyclic  for the  direct sum operator $\oplus_{i=1}^N T_i$ acting on the space $X^N=X\times  \cdots\times X$, endowed with the product topology.\end{definition}

It is well-known that two d-hypercyclic operators must be substantially different (see, e.g. \cite{BP}). For instance, an operator can not be d-hypercyclic with a scalar multiple of itself. For recent related results on disjointness  in hypercyclicity, see, e.g.\cite{B,BMS,BP,HL,LX} and their references therein.

The following \emph{Disjoint-Hypercyclicity Criterion} (or \emph{d-Hypercyclicity Criterion} for short) is an extension of the \emph{Hypercyclicity Criterion} (see, e.g. \cite{Ki}), which is a sufficient condition ensuring the d-hypercyclicity of finitely many hypercyclic operators.

 \begin{definition}\label{defn DHCC}\cite[Definition 2.5]{BP}  Let $(n_k) $ be an increasing sequence of positive integers. We say that $N\geq 2$ operators $T_1,\cdots,T_N$ in  $L(X)$ satisfy the \emph{d-Hypercyclicity Criterion} with respect to $(n_k) $ provided that there
 exist dense subsets $X_0,\cdots,X_N$ of $X$ and mappings $S_{l,k}:X_l\rightarrow X \; (k\in \mathbb{N},\; 1\leq l\leq N)$ satisfying
 \begin{eqnarray*}
 &&(i)\;\; T^{n_k}_l\mathop \to \limits_{k \to \infty} 0 \;\mbox{ pointwise  on }\; X_0,\\
 &&(ii)\;\; S_{l,k} \mathop \to \limits_{k \to \infty} 0 \;\mbox{ pointwise  on }\; X_l,\;\mbox{and}\\
 &&(iii) \;\;(T_l^{n_k} S_{i,k} -\delta_{i,l} Id_{X_l})\mathop \to \limits_{k \to \infty} 0   \;\mbox{ pointwise  on }\; X_l\; (1\leq i\leq N). \end{eqnarray*}
 In general, we say that $T_1,\cdots,T_N$ satisfy the \emph{d-Hypercyclicity Criterion} if there exists some  increasing sequence of positive integers  $(n_k)$ for which the above conditions are satisfied. \end{definition}
If  $X$ is a Fr\'{e}chet (therefore, a Baire) space, we have the following proposition.
 \begin{proposition}\label{prop DHCC}\cite[Proposition 1.3.7]{M1} Let $N\geq2$ and $T_1,...,T_N\in L(X)$ satisfy the d-Hypercyclicity Criterion.
Then the set of d-hypercyclic vectors for $T_1,...,T_N$ is a dense $G_\delta$ set of $X$.  \end{proposition} Analogously to the topologically transitive, the \emph{d-topologically transitive} for operators $T_1,\cdots, T_N$ is defined as below.
\begin{definition}\label{defn d-transitive}\cite[Definition 2.1]{BP} We say that  $N\geq 2$ operators $T_1,\cdots, T_N$ in $L(X)$ are \emph{d-topologically transitive} provided for every non-empty open subsets $V_0,\cdots,V_N$ of $X$ there exists $m\in \mathbb{N}$ such that $$V_0\cap T_1^{-m}(V_1)\cap\cdots\cap T_N^{-m}(V_N)\neq \emptyset.$$\end{definition} Concerning the relationship between d-topologically transitive and d-hypercyclic, it follows a proposition.
\begin{proposition}\label{prop d-transitive}\cite[proposition 2.3]{BP} Let $T_1,\cdots,T_N$ in $L(X).$ Then the following are equivalent:

$(a)$ $T_1,\cdots,T_N$ are d-topologically transitive,

$(b)$ the set of d-hypercyclic vectors for $T_1,\cdots,T_N$ is a dense $G_\delta$ subset of $X$. \end{proposition}

\section{Main result}

For $1\leq p<\infty$, we will find an equivalent condition to characterize finite set of weighted translations $T_{a,w_1},\cdots , T_{a,w_N}$ on $L^p(G)$ sharing a dense $G_\delta$ set of d-hypercyclic vectors. Of course, our theorem provides a sufficient condition for d-hypercyclicity. For different type descriptions for d-hypercyclicity and d-supercyclicity of weighted translations on $L^p(G)$, the interested readers can refer to our recent papers \cite{HL,LX}.

\begin{theorem}\label{thm disjoint}Let $G$ be a second countable locally compact group and $a$ be an aperiodic element in $G$. Let $1\leq p<\infty$ and $T_{a,w_1},\cdots , T_{a,w_N}$ be $N\geq 2$ hypercyclic  weighted translations on   $L^p(G)$, where $w_i$ is a weight on $G$ for $i=1, \cdots, N$. Then the following statements are equivalent:

$(1)$\;  $T_{a,w_1},\cdots , T_{a,w_N}$ have a dense $G_\delta$ set of d-hypercyclic vectors.

$(2)$\; For each compact subset $K\subset G$ with $\lambda (K)>0$,  there is a sequence of Borel sets $(E_k)$ in $K$ such that $\lambda (K)=\lim\limits_{k\rightarrow \infty} \lambda(E_k)$, and  there exists an increasing subsequence $(n_k)\subset \mathbb{N}$ satisfying
\begin{eqnarray}&&\;\mbox{if}\;\;1\leq l\leq N,\;\left\{
                                                  \begin{array}{ll}
                                                  (i)\;\; \lim\limits_{k\rightarrow \infty} \left\|\prod_{s=1}^{n_k} w_{l}*\delta_{a^{-1}}^s\Big|_{E_k}\right\|_\infty=0,  \vspace{1mm}\\
                                                (ii)\;\lim\limits_{k\rightarrow\infty}\left\|\left(\prod_{s=0}^{n_k-1} w_{l}*\delta_{a}^s\right)^{-1}\Big|_{E_k}\right\|_\infty=0,
                                                  \end{array}
                                                \right.\label{3.1}\\
&&\;\mbox{if}\;\;1\leq l,\;j\leq N\;\; \mbox{with}\;\;l\neq j,\;\lim\limits_{k\rightarrow \infty}\left\| \left( \prod_{s=0}^{n_k-1} \frac{w_j*\delta_a^s(x)}{ w_l*\delta_a^s(x)}\right)\Big|_{E_k}\right\|_{\infty}=0.\label{3.2}
\end{eqnarray}
Either the condition in $(1)$ or $(2)$ holds, $T_{a,w_1},\cdots , T_{a,w_N}$ are d-hypercyclic.
\end{theorem}

\begin{proof} By Lemma \ref{lem aperiodic}, since $a$ is an \emph{aperiodic} element in $G$, there exists $M\in \mathbb{N}$ such that $K\cap Ka^{\pm n}=\emptyset$ for $n>M.$

$(2)\Rightarrow (1).$  Suppose that $(2)$ holds, we will show that the weighted translations $T_{a,w_1},\cdots , T_{a,w_N}$ are \emph{d-topologically transitive}, that is, they satisfy the Definition \ref{defn d-transitive}. Let $V_0, V_1 \cdots, V_N$ be non-empty open subsets of $L^p(G).$  Since the space $C_{c} (G)$ of continuous functions on $G$ with compact support is dense in $L^p(G)$, then there exist  $f, f_1,\cdots, f_N\in C_c(G)$ satisfying $f\in V_0, f_1\in V_1,\cdots, f_N\in V_N.$  Let $K$ be the union of the compact supports of $f,f_1,\cdots,f_N$ and let $\chi_K\in L^p(G)$ be the characteristic function of $K$. Suppose that $(E_k)_{k}$ is a sequence of Borel sets  in $K$ satisfying $\lambda (K)=\lim\limits_{k\rightarrow \infty} \lambda(E_k)$. Let $\epsilon>0$ such that every open $L^p(G)$-ball $B(f_j,\epsilon)$ is contained in $V_j$ for $j=0,1,\cdots, N$. By conditions in $(2)$, there exists $m\in \mathbb{N}$ such that $k>m$ and $n_k>M$, it holds that for  $1\leq l\leq N$,
\begin{eqnarray} &&\prod_{s=1}^{n_k} w_l*\delta_{a^{-1}}^s(x)^p<\frac{\epsilon}{3\|f\|_p^p}, \;\;\prod_{s=0}^{n_k-1} w_{l}*\delta_{a}^s(x) ^{-p} <\frac{\epsilon}{3N\|f_l\|_p^p} \;\;\mbox{on}\;E_k;\label{fs} \\&& \lambda(K \setminus E_k)<\frac{\epsilon}{3\max\{\|f\|_\infty^p,\|f_1\|_\infty^p,\cdots,\|f_N\|_\infty^p\}},\label{lam}\end{eqnarray} and  $\mbox{ for}\;\;1\leq l,\;j\leq N\; \mbox{with}\;\;l\neq j,\;$   \begin{eqnarray}   &&\left( \prod_{s=0}^{n_k-1} \frac{w_j*\delta_a^s(x)}{ w_l*\delta_a^s(x)}\right)^p<\frac{\epsilon}{3N\|f_l\|_p^p}\;\;\mbox{on}\;E_k.\label{TS}\end{eqnarray}
By the first inequality in \eqref{fs} it follows for every $1\leq l\leq N$ that
\begin{eqnarray} \|T_{a,w_l}^{n_k}(f\chi_{E_k})\|_p^p&=&\int_{G} |T_{a,w_l}^{n_k}(f\chi_{E_k})(x)|^p d\lambda(x)\nonumber\\&=&
\int_{E_k a^{n_k}} |w_l(x)w_l(xa^{-1})\cdots w_l(xa^{-(n_k-1)})|^p|f(xa^{-n_k})|^p d\lambda(x)
\nonumber\\&=&\int_{E_k}|w_l(xa^{n_k})w_l(xa^{n_k-1}) \cdots w_l(x a)|^p |f(x)|^p d\lambda(x)
\nonumber\\&=& \int_{E_k} \left| \prod_{s=1}^{n_k} w_l*\delta_{a^{-1}}^s(x)\right|^p |f(x)|^p d\lambda(x) <\frac{\epsilon}{3},\label{T}\end{eqnarray}
for all $k>m.$ 
For each $1\leq l\leq N$, we define the self-map $S_{a,w_l}$ on the subspace $L_c^p(G)$ of functions in $L^p(G)$ with \emph{compact support} by
$$ S_{a,w_l}(h)=\frac{h}{w_l}* \delta_{a^{-1}}.$$ And let
 $S^{n_k}_{a,w_l}=\underbrace{S_{a,w_l}\circ\cdots \circ S_{a,w_l}}\limits_{n_k}.$   It is obvious that $$T_{a,w_l}^{n_k}S_{a,w_l}^{n_k}(h)=h.$$    Then by the second inequality in \eqref{fs}, it turns out that
\begin{eqnarray}\|S_{a,w_l}^{n_k} (f_l\chi_{E_k})\|_p^p&=&\int_{G}|S_{a,w_l}^{n_k} (f_l\chi_{E_k})(x)|d\lambda(x)
\nonumber\\&=&\int_{E_k a^{-n_k}} \frac{1}{|w_l(xa)w_l(xa^2)\cdots w_l(xa^{n_k})|^p} |f_l(x a^{n_k})|^p d\lambda(x)\nonumber\\&=&
\int_{E_k}\frac{1}{|w_l(x a^{1-n_k}) w_l(x a^{2-n_k})\cdots w_l(x)|^p}|f_l(x)|^p d\lambda(x)\nonumber\\&=&
\int_{E_k} \left| \prod_{s=0}^{n_k-1} w_l*\delta_a^s(x)\right|^{-p}  |f_l(x)|^p d\lambda(x)<\frac{\epsilon}{3N},\;\; \label{S}\end{eqnarray}
for $1\leq l\leq N$ and $k>m.$
When $1\leq j\neq l\leq N$, by \eqref{TS}  we have that
\begin{eqnarray}\|T_{a,w_j}^{n_k} S_{a,w_l}^{n_k} (f_l\chi_{E_k})\|_p^p&=&\int_{G}|T_{a,w_j}^{n_k} S_{a,w_l}^{n_k} (f_l\chi_{E_k})(x)|^pd\lambda(x) \nonumber\\&=&\int_{G} \left| T_{a,w_j}^{n_k} \frac{f_l\chi_{E_k}(x a^{n_k})}{ w_l(xa)w_l(xa^2)\cdots w_l(xa^{n_k}) } \right|^p d\lambda(x)\\&=& \int_{G}\left| \frac{w_j(x)w_j(xa^{-1})\cdots w_{j} (xa^{1-n_k})  f_l\chi_{E_k}(x)}{ w_l(x a^{1-n_k}) w_l(x a^{2-n_k})\cdots w_l(x) }\right|^p d\lambda(x)\nonumber\\&=&   \int_{G}\left| \frac{w_j(x)w_j(xa^{-1})\cdots w_{j} (xa^{1-n_k}) }{ w_l(x)w_l(xa^{-1})\cdots w_l(x a^{1-n_k})}\right|^p | f_l\chi_{E_k}(x)|^p d\lambda(x) \nonumber \\&=&\int_{E_k}  \left| \prod_{s=0}^{n_k-1} \frac{w_j*\delta_a^s(x)}{ w_l*\delta_a^s(x)}\right|^{p}  |f_l(x)|^p d\lambda(x)<\frac{\epsilon}{3N},  \label{TS1} \end{eqnarray} for $k>m.$ 
 For each $k\in \mathbb{N},$ let
 $$u_k=f\chi_{E_k} +\sum_{l=1}^N  S_{a,w_l}^{n_k}(f_{l}\chi_{E_k} )\in L^p(G).$$ For $k>m$, we combine the  displays \eqref{lam} and \eqref{S} to deduce that
 \begin{eqnarray*}\|u_k-f\|_p^p&\leq& \|f\|_\infty^p \lambda(K\setminus E_k)+ \sum_{l=1}^N \| S_{a,w_l}^{n_k}(f_{l}\chi_{E_k} )\|_p^p\\&<&\frac{\epsilon}{3}+\frac{N\epsilon}{3N}<\epsilon.\end{eqnarray*} And then employing \eqref{T}, \eqref{lam} and \eqref{TS1}, we obtain that \begin{eqnarray*}
 \|T_{a,w_j}^{n_k} u_k-f_j\|_p^p&=&\|T_{a,w_j}^{n_k}(f\chi_{E_k} +\sum_{l=1}^N  S_{a,w_l}^{n_k}(f_{l}\chi_{E_k} ))-f_j\|_p^p\\&\leq & \|T_{a,w_j}^{n_k}(f\chi_{E_k})\|_p^p+ \|f_j\|_\infty^p\lambda(K\setminus E_k)+ \sum_{l\neq j}\| T_{a,w_j}^{n_k}S_{a,w_l}^{n_k} (f_{l}\chi_{E_k} ) \|_p^p\\&<&\frac{\epsilon}{3}+\frac{\epsilon}{3}+\frac{(N-1)\epsilon}{3N}< \epsilon.\end{eqnarray*}
That is, $$u_k\in V_0\cap T_{a,w_1}^{-n_k} (V_1)\cap \cdots T_{a,w_N}^{-n_k}(V_N)\neq \emptyset,$$ for $k>m$.  Hence $T_{a,w_1},\cdots , T_{a,w_N}$ satisfy Definition \ref{defn d-transitive}. Further by Proposition \ref{prop d-transitive}, we obtain $(1)$.\vspace{2mm}

$(1)\Rightarrow (2)$. Suppose that  $T_{a,w_1},\cdots , T_{a,w_N}$ have a dense $G_\delta$ set of \emph{d-hypercyclic} vectors. Let $K\subset G$ be a compact set with $\lambda(K)>0$ and $\chi_K\in L^p(G)$ be the characteristic function of $K$.    Let $\epsilon>0$ and choose $\eta\in(0,1)$ satisfying $0<\frac{\eta}{1-\eta}<\epsilon.$ By the density of \emph{d-hypercyclic} vectors, there exist $f\in L^p(G)$ and some integer $m>M$ such that
\begin{eqnarray} &&\|f-\chi_K\|_p<\eta^2, \label{3.6}\\&&\|T_{a,w_l}^{m}f-\chi_K\|_p<\eta^2,\;\; 1\leq l\leq N.\label{3.7} \end{eqnarray}
On the one hand, denote the set  $$A_\eta=\{x\in K: |f(x)-1|\geq \eta\}.$$ Then by \eqref{3.6}
 \begin{eqnarray*} \eta^{2p}>\|f-\chi_K\|_p^p\geq \int_{A_\eta} |f(x)-1|^p d\lambda(x) \geq \eta^p \lambda(A_\eta),   \end{eqnarray*}
 from which we get $\lambda(A_\eta)<\eta^p,$ and it holds
 \begin{eqnarray} |f(x)|>1-\eta,\;\mbox{for}\;\;x\in K\setminus A_\eta.\label{3.8} \end{eqnarray}
On the other hand, denote the set $$B_\eta=\{x\in G \setminus K:\; | f(x)|\geq \eta\}.$$ Analogously, it follows that
 \begin{eqnarray*}\eta^{2p}>\|f-\chi_K\|_p^p\geq \int_{G\setminus K} |f(x)-\chi_K(x)|^p d\lambda(x)\geq \int_{B_\eta}|f(x)|^p d\lambda(x)\geq \eta^p\lambda(B_\eta),  \end{eqnarray*} which implies that $\lambda(B_\eta) <\eta^p,$ and it is true that
 \begin{eqnarray}|f(x)|<\eta,\;\;\mbox{for}\;\;x\in (G\setminus K)\setminus B_\eta.\label{3.9} \end{eqnarray}
 Denote $$C_{l,m,\eta}=\{x\in K:\; \left|\prod_{s=0}^{m-1} w_l*\delta_a^s(x) f(xa^{-m})-1\right|\geq \eta\}$$ for $1\leq l\leq N.$ Then by \eqref{3.7} we get that
\begin{eqnarray*}&&\eta^{2p}> \|T_{a,w_l}^{m}f-\chi_K\|_p^p\\&&\geq \int_{C_{l,m,\eta}}|w_l(x)w_l(xa^{-1})\cdots w_l(xa^{-(m-1)})f(xa^{-m})-1|^pd\lambda(x)\\&&\geq \eta^p\lambda(C_{l,m,\eta}).\end{eqnarray*} Therefore, $\lambda(C_{l,m,\eta})<\eta^p$ for $1\leq l\leq N,$ and it is trivial that
\begin{eqnarray} \left|\prod_{s=0}^{m-1} w_l\ast \delta_a^s(x) f(xa^{-m})\right|>1-\eta,\;x\in K\setminus C_{l,m,\eta},\;\mbox{for}\;\; 1\leq l\leq N.\label{3.10} \end{eqnarray} Since $a$ is an \emph{aperiodic} element in $G$ and $m> M,$ we have $K\cap Ka^{-m}=\emptyset,$ and so we obtain for $1\leq l\leq N$ that
\begin{eqnarray}\left( \prod_{s=0}^{m-1} w_l\ast \delta_a^s(x)\right)^{-1}<\frac{|f(xa^{-m})|}{1-\eta}<\frac{\eta}{1-\eta}<\epsilon,\;x\in K\setminus (a^m B_\eta \cup C_{l,m,\eta}). \label{3.11}\end{eqnarray}
Let $$D_{l,m,\eta}=\{x\in K:\;\left|\prod_{s=1}^{m} w_{l}*\delta_{a^{-1}}^s(x) f(x) \right|\geq \eta\}$$ for $1\leq l\leq N.$ By \eqref{3.7}, the right-invariance of the Haar measure $\lambda$ and the fact $K\cap Ka^{m}=\emptyset,$ we infer that
\begin{eqnarray*}&&\eta^{2p}> \|T_{a,w_l}^{m}f-\chi_K\|_p^p\\&&= \int_{G}|w_l(x)w_l(xa^{-1})\cdots w_l(xa^{-(m-1)})f(xa^{-m})-\chi_K(x)|^pd\lambda(x)\\&&=\int_{G}|w_l(xa^m)w_l(xa^{m-1})\cdots w_l(xa)f(x)-\chi_K(xa^m)|^pd\lambda(x)\\&&\geq \int_{D_{l,m,\eta}}|w_l(xa^m)w_l(xa^{m-1})\cdots w_l(xa)f(x)|^pd\lambda(x)\\&&\geq  \eta^p \lambda(D_{l,m\eta}),
 \end{eqnarray*} from which it follows that $\lambda(D_{l,m\eta})<\eta^p$ for $1\leq l\leq N$ and it holds that
 \begin{eqnarray}\prod_{s=1}^{m} w_{l}*\delta_{a^{-1}}^s(x)<\frac{\eta}{|f(x)|}<\frac{\eta}{1-\eta}<\epsilon,\; x\in K\setminus(A_\eta\cup D_{l,m,\eta}),\;1\leq l\leq N. \label{3.12}\end{eqnarray}
Let the sets $$F_{l,m,\eta}=\{x\in G\setminus K: \left|\prod_{s=0}^{m-1} w_l*\delta_a^s(x) f(xa^{-m})\right|\geq \eta \}$$ for $1\leq l\leq N.$  By \eqref{3.7} it is clear that
\begin{eqnarray*}&&\eta^{2p}>\|T_{a,w_l}^m f-\chi_K\|_p^p \geq \int_{G\setminus K}|w_l(x)w_l(xa^{-1})\cdots w_l(xa^{-(m-1)}) f(xa^{-m})|^pd\lambda(x)\\&&\geq \int_{F_{l,m,\eta}} |w_l(x)w_l(xa^{-1})\cdots w_l(xa^{-(m-1)}) f(xa^{-m})|^pd\lambda(x) \geq \eta^p \lambda(F_{l,m,\eta}),\end{eqnarray*} and from above it follows that $\lambda(F_{l,m,\eta}) <\eta^p$ for $1\leq l\leq N$.\vspace{3mm}

 For $x\in K\setminus (C_{l,m,\eta}\cup F_{j,m,\eta}),$ where $l\neq j,$ by the definition of $F_{j,m,\eta}$   we get that $$|w_j(x)w_j(xa^{-1})\cdots w_j(xa^{-(m-1)}) f(xa^{-m})|< \eta,\;1\leq j\leq N.$$ Further by \eqref{3.10}, for $1\leq l,\;j\leq N$ with $l\neq j$, it follows that
\begin{eqnarray} \prod_{s=0}^{m-1} \frac{w_j*\delta_a^s(x)}{ w_l*\delta_a^s(x)} &=&\frac{w_j(x)w_j(xa^{-1})\cdots w_j(xa^{-(m-1)})f(xa^{-m})}{w_l(x)w_l(xa^{-1})\cdots w_l(xa^{-(m-1)})f(xa^{-m})}\nonumber\\&<&\frac{\eta}{1-\eta}<\epsilon.\label{3.13} \end{eqnarray}
Denote the sets $$C_{m,\eta}=\cup_{ l=1}^N C_{l,m,\eta}, \; D_{m,\eta}=\cup_{l=1}^N D_{l,m,\eta}\;\mbox{and}\;F_{m,\eta}=\cup_{ l=1}^N F_{l,m,\eta}.$$ Further denote $H_{m,\eta}=C_{m,\eta}\cup D_{m,\eta}\cup F_{m,\eta}.$  It is trivial that $\lambda (H_{m,\eta})\leq 3N\eta^p.$ Now let $$E_{m,\eta}=K\setminus (A_\eta\cup a^m B_\eta\cup H_{m,\eta})$$ with $\lambda(A_\eta\cup a^m B_\eta\cup H_{m,\eta})<(2+3N)\eta^p<(2+3N)\epsilon^p$,
Hence for $x\in E_{m,\eta}$, by \eqref{3.12}, \eqref{3.11} and \eqref{3.13} we respectively get that
\begin{eqnarray*}&&\prod_{s=1}^{m} w_{l}*\delta_{a^{-1}}^s(x)<\epsilon,\; \left( \prod_{s=0}^{m-1} w_l\ast \delta_a^s(x)\right)^{-1}<\epsilon,\;\;\prod_{s=0}^{m-1} \frac{w_j*\delta_a^s(x)}{ w_l*\delta_a^s(x)} <\epsilon \;(j\neq l). \end{eqnarray*}
Thus for $k=1,2,\cdots,$ the above inequalities enable us to construct a Borel set $E_k\subset K$ and find an increasing sequence  $(n_k)\subset \mathbb{N}$  such that $\lambda(K\setminus E_k)\leq C (1/k)^p$, $C$ independent of $k,$ and \begin{eqnarray*} &&\mbox{if}\;\;1\leq l\leq N,\;\;\left\|\prod_{s=1}^{n_k} w_{l}*\delta_{a^{-1}}^s\Big|_{E_k}\right\|_\infty<\epsilon,  \;\;\left\|\left(\prod_{s=0}^{n_k-1} w_{l}*\delta_{a}^s\right)^{-1}\Big|_{E_k}\right\|_\infty<\epsilon,                                            \\&& \;\mbox{if}\;\;1\leq l,\;j\leq N\;\; \mbox{with}\;\;l\neq j,\;\; \left\| \left( \prod_{s=0}^{n_k-1} \frac{w_j*\delta_a^s(x)}{ w_l*\delta_a^s(x)}\right)\Big|_{E_k}\right\|_{\infty}<\epsilon.\end{eqnarray*} That means  the result $(2)$ holds.  Besides, under either  the condition in $(1)$ or $(2)$, we can conclude $T_{a,w_1},\cdots , T_{a,w_N}$ are \emph{d-hypercyclic}. This completes the proof.
\end{proof}

\bibliographystyle{amsplain}

\end{document}